\documentclass[12pt]{amsart}
\begin{document}

\newtheorem{theorem}{Theorem}[section]
\newtheorem{lemma}[theorem]{Lemma}
\newtheorem{corollary}[theorem]{Corollary}
\theoremstyle{definition}
\newtheorem{definition}[theorem]{Definition}
\theoremstyle{remark}
\newtheorem{remark}[theorem]{Remark}
\newcommand{\Z}{\mathbb{Z}}
\newcommand{\D}{\mathbb{D}}
\newcommand{\R}{\mathbb{R}}
\newcommand{\N}{\mathbb{N}}
\newcommand{\C}{\mathbb{C}}
\newcommand{\card}{\operatorname{card}}
\def\capacity{\operatorname{cap}}
\def\diam{\operatorname{diam}}
\renewcommand\Re{\operatorname{Re}}
\renewcommand\Im{\operatorname{Im}}
\numberwithin{equation}{section}
\title[Julia set and fast escaping
set]{The Julia set and the fast escaping set of a quasiregular mapping}
\dedicatory{Dedicated to the memory of Professor Frederick W.\ Gehring}

\author{Walter Bergweiler}
\address{Mathematisches Seminar,
Christian--Albrechts--Universit\"at zu Kiel,
Lude\-wig--Meyn--Str.~4,
D--24098 Kiel,
Germany}
\email{bergweiler@math.uni-kiel.de}
\author{Alastair Fletcher}
\address{Department of Mathematical Sciences, Northern Illinois University,
DeKalb, IL 60115-2888, USA}
\email{fletcher@math.niu.edu}
\author{Daniel A. Nicks}
\address{School of Mathematical Sciences, University of Nottingham, Nottingham NG7 2RD, UK}
\email{Dan.Nicks@nottingham.ac.uk}
%\date{\today}
\thanks{This work was carried out during a visit of the second and third authors to Christian--Albrechts--Universit\"at zu Kiel, which they would like to thank for its hospitality.} 
\begin{abstract}
It is shown that for quasiregular maps of positive lower order
the Julia set coincides with the boundary of the fast escaping
set.
\end{abstract}
\maketitle

\section{Introduction}
Quasiregular maps are a natural generalisation of holomorphic functions
to higher dimensions, see section~\ref{quasi} for the definition and basic
properties.
A quasiregular self-map $f$ of the $d$-sphere $S^d$ is called
\emph{uniformly quasiregular} if there is a uniform bound on the
dilatation of the iterates $f^n$ of $f$.
As in the case of rational functions, the \emph{Julia set} of such
a map is defined as the set of all points in $S^d$ where
the iterates of $f$ fail to be normal. Many results of the
Fatou-Julia iteration theory of rational functions continue to hold in
this more general setting; see~\cite[Section 4]{Bergw2010}
and~\cite[Chapter~21]{Iwaniec01} for surveys.
In principle the corresponding iteration theory of transcendental entire
functions could also be extended to uniformly quasiregular self-maps of $\R^d$
with an essential singularity at~$\infty$,
but so far no examples of such maps are known for $d\geq 3$.
And for $d=2$ uniformly quasiregular
maps are conjugate to holomorphic maps; see \cite[Section 4.1]{Bergw2010}
for a discussion of this result.

A Fatou-Julia theory for quasiregular self-maps of $S^d$ which are not
uniformly quasiregular was developed in \cite{Bergw2013}, extending
work of Sun and Yang~\cite{Sun99,Sun00,Sun01} dealing with the case $d=2$.
The theory was carried over to the case of
quasiregular self-maps of $\R^d$ with an essential
singularity at $\infty$ in~\cite{BN}. Such maps are said to be of
\emph{transcendental type}. In~\cite{Bergw2013,BN} the Julia set
 $J(f)$ of $f$ is defined as the set of all $x$ such that
the complement of $\bigcup^\infty_{k=1}f^k(U)$ has capacity zero
for every neighbourhood $U$ of~$x$;
see section~\ref{quasi} for the definition and a discussion of capacity.

The \emph{escaping set}
$I(f)$ of a quasiregular self-map $f$ of $\R^d$ is defined by
\[ I(f) = \{ x \in \R^d \colon f^n(x) \to \infty \},\]
and it plays a major role in the dynamics of entire functions.
The escaping set was first considered by Eremenko~\cite{Eremenko89}
for transcendental entire functions. He showed that it is always non-empty.
The result was extended to quasiregular maps of transcendental type
in~\cite{BFLM}.
Moreover, Eremenko proved
that $J(f)=\partial I(f)$ for every transcendental entire function.
This key result in the dynamics of entire functions fails to be true
in general for quasiregular maps of transcendental type. In this setting
we still have $J(f)\subset \partial I(f)$, but strict inclusion is
possible~\cite[Theorem~1.3, Example~7.3]{BN}.

Besides the escaping set $I(f)$,
the \emph{fast escaping set} $A(f)$ introduced in~\cite{BH}
has become increasingly important in recent years, see~\cite{RS1,RS}.
In order to define it, recall that the \emph{maximum modulus} is given by
\[
M(r,f)=\max_{|x|=r}|f(x)|,
\]
where $|x|$ denotes the Euclidean norm of a point $x\in\R^d$.
If $f$ is a quasiregular map of transcendental type, then
there exists $R_0\geq 0$ such that $M(R,f)>R$ for $R>R_0$.
Denote by $M^n(R,f)$ the iteration of $M(R,f)$
with respect to the first variable; that is,
$M^1(R,f) = M(R,f)$ and $M^n(R,f) = M(M^{n-1}(R,f),f)$ for $n\geq 2$.
For $R>R_0$ we then have $M^n(R,f)\to\infty$ as $n\to\infty$.
For such $R$ we put
\begin{equation}
\label{1b}
A(f) = \{ x \in \R^d \colon \exists L \in \N\;\forall n\in\N\colon |f^{n+L}(x)| \geq M^n(R,f)\}.
\end{equation}
It can be shown that this definition does not depend on $R$ as long as
$M^n(R,f)\to\infty$, see \cite{RS} for transcendental entire functions
and~\cite{BDF} for quasiregular maps. These papers also contain some
equivalent definitions of $A(f)$.

The methods used in~\cite{BN} to show that $J(f)\subset \partial I(f)$
also yield the following result.
\begin{theorem}
\label{thm1}
Let $f\colon\R^d \to \R^d$ be a quasiregular map of transcendental type.
Then $J(f)\subset \partial A(f)$.
\end{theorem}

For transcendental entire functions we have not only $J(f)=\partial I(f)$,
but also $J(f)=\partial A(f)$, see~\cite[p.~1125, Remark~1]{RS1}.
While, as mentioned, the first equation does not hold for quasiregular
maps in general, we conjecture that the second equation remains valid
for quasiregular maps of transcendental type.
We can show that this is the case for maps which do not grow too slowly.
\begin{theorem}
\label{thm2}
Let $f\colon\R^d \to \R^d$ be a quasiregular map
satisfying
\begin{equation}\label{1a}
\liminf_{r\to\infty}\frac{\log\log M(r,f)}{\log \log r}=\infty.
\end{equation}
Then $J(f)=\partial A(f)$.
\end{theorem}
Recall that the \emph{lower order} of a quasiregular self-map of $\R^d$ is given
by~\cite[Section V.8]{Rickman}
\[
(d-1)\liminf_{r\to\infty}\frac{\log\log M(r,f)}{\log r}.
\]
Thus Theorem~\ref{thm2} applies in particular to functions of  positive lower order.
For example, the result applies to quasiregular analogues of the
exponential function and the trigonometric functions whose dynamics
where studied in~\cite{Bergw2010a,BE,FN3}.
On the other hand, quasiregular maps of transcendental type
which have lower order zero were constructed in~\cite{DS}.

One consequence of Theorem~\ref{thm2} is that if
$f\colon\R^d\to\R^d$  is a quasiregular map satisfying~\eqref{1a}
and if $n\in\N$, then $J(f^n)=J(f)$.
This follows since $A(f^n)=A(f)$ for all quasiregular maps
of transcendental type~\cite[Proposition 3.1, (ii)]{BDF}.

A quasiregular map $f\colon\R^d\to\R^d$ which is not of transcendental
type is said to be of \emph{polynomial type}. It was shown in~\cite{FN}
that if $f$ is such a map satisfying $\deg f>K_I(f)$, then $I(f)\neq\emptyset$
and $\partial I(f)$ is perfect. Moreover, $I(f)=A(f)$ for such
maps~\cite{FN4}.
The inclusion $J(f)\subset\partial I(f)$ also holds for such maps,
and again it may be strict; cf.~\cite{Bergw2013,Nicks2013}.

\section{Quasiregular maps} \label{quasi}
We introduce quasiregular maps only briefly and refer to~\cite{Rickman}
for more details.
Let $\Omega\subset\R^d$ be a domain and $1\leq p<\infty$.
The \emph{Sobolev space}  $W^1_{p,\text{loc}}(\Omega)$ consists of
all functions  $f=(f_1,\dots,f_d)\colon \Omega\to \R^d$ for which
all first order weak partial derivatives
$\partial_k f_j$ exist and are locally in $L^p$.
A~continuous map $f\in W^1_{d,\text{loc}}(\Omega)$
is called \emph{quasi\-regular} if there exists a constant $K_O\geq 1$
such that \begin{equation}\label{2a}
|Df(x)|^d\leq K_O J_f(x)
\quad \mbox{ a.e.}
\end{equation}
Here $Df(x)$ is the derivative of $f$ at $x$,
\[
|Df(x)|=\sup_{|h|=1} |Df(x)(h)|
\]
its norm, and $J_f(x)$ the Jacobian determinant.
Put
\[
\ell(Df(x))=\inf_{|h|=1}|Df(x)(h)|.
\]
It turns out that, for $f\in W^1_{d,\text{loc}}(\Omega)$,
the condition that~\eqref{2a} holds
for some $K_O\geq 1$ is equivalent to the existence of a constant $K_I\geq 1$
such that
\begin{equation}\label{2b}
J_f(x)\leq K_I\ell(Df(x))^d
\quad \mbox{ a.e.}
\end{equation}
The smallest constants
$K_O$ and $K_I$ such that~\eqref{2a} and~\eqref{2b} hold are
denoted by $K_O(f)$ and $K_I(f)$ and
called the \emph{outer and inner dilatation} of $f$.

For quasiregular maps $f,g\colon\R^d\to\R^d$ we have~\cite[Theorem~II.6.8]{Rickman}
\[
K_O(f\circ g)\leq K_O(f)K_O(g)
\quad\text{and}\quad
K_I(f\circ g)\leq K_I(f)K_I(g)
\]
and thus
\begin{equation}\label{Kfn}
K_O(f^n)\leq K_O(f)^n
\quad\text{and}\quad
K_I(f^n)\leq K_I(f)^n
\end{equation}
for all $n\in\N$.

Let $G\subset\R^d$ be open and $C\subset G$ compact.
The pair $(G,C)$ is called a \emph{condenser} and
the quantity
\[
\capacity (G,C)=\inf_u\int_G\left|\nabla u\right|^d dm,
\]
with the infimum taken over all non-negative functions
$u\in C^\infty_0(G)$ satisfying $u(x)\geq 1$ for $x\in C$,
is called the \emph{capacity} of the condenser.
If $\capacity (G,C)=0$
for some bounded open set $G$ containing~$C$,
then  $\capacity (G',C)=0$
for every bounded open set $G'$ containing~$C$;
see \cite[Lemma III.2.2]{Rickman}.
In this case we say that $C$ is of  \emph{capacity zero}, while
otherwise $C$ has \emph{positive capacity}.
A closed subset $C$ of $\R^d$ is said to have capacity
zero if this is the case for all compact subsets of~$C$.

An important concept used in the study of quasiregular maps is
the modulus of a curve family;
see~\cite[Chapter II]{Rickman} and~\cite[Chapter~2]{Vuorinen1988}
for a detailed discussion.
For a family $\Gamma$ of paths in $\R^d$, a
non-negative Borel function $\rho\colon \R^d\to\R\cup\{\infty\}$
is called \emph{admissible} if $\int_\gamma \rho\; ds\geq 1$ for all
locally rectifiable paths $\gamma\in \Gamma$. Let
$\mathcal{F}(\Gamma)$ be the family
of all admissible Borel functions. Then
$$
M(\Gamma)=\inf_{\rho\in \mathcal{F}(\Gamma)}\int_{\R^d}\rho^d\; dm
$$
is called the \emph{modulus} of $\Gamma$.

For a domain $G\subset \R^d$ and subsets $E,F$ of $\overline{G}$,
let $\Delta(E,F;G)$ be the family of all paths which have
one endpoint in~$E$ and one endpoint in~$F$, and which are in $G$
otherwise.
A  connection between capacity and the modulus of
a path family was first noted by Gehring~\cite{G} and this has since been extended to
the general equation~\cite[Proposition II.10.2]{Rickman}
$$\capacity (G,C)=M(\Delta(C,\partial G;G)).$$

We denote by $B(a,r)$ the open ball  and by $\overline{B}(a,r)$
the closed ball of radius $r$ around a point $a\in\R^d$ and
by $\diam X$ the Euclidean diameter of a subset $X$ of $\R^d$.
We need the following estimate of Vuorinen
(\cite[Lemma~2.18]{Vuorinen1983}, \cite[Lemma~5.42]{Vuorinen1988}).
\begin{lemma}\label{lemma4}
Let $a\in \R^d$, $r>0$ and $E,F$ continua in $\overline{B}(a,r)$.
Then
\[
M(\Delta(E,F;B(a,r))
\geq c_0 \frac{\min\{\diam E,\diam F\}}{r}
\]
for some constant $c_0$ depending only on~$d$.
\end{lemma}

\section{Averages of counting functions}\label{averages}
For a quasi\-regular map $f\colon \Omega\to \R^d$ and $x\in \Omega$
the \emph{local index} $i(x,f)$
of $f$ at $x$
is defined by
$$i(x,f) =\inf_U \sup_{y\in\R^d}\card \left( f^{-1}(y)\cap U\right),$$
the infimum being taken over all neighborhoods $U\subset \Omega$ of~$x$.
For $y\in\R^d$ and $E\subset \Omega$ compact
we put
\[
n(E,y)=\sum_{x\in f^{-1}(y) \cap E} i(x,f).
\]
Thus $n(E,y)$ is the number of $y$-points of $f$ in~$E$,
counted according to multiplicity.

The cardinality of $f^{-1}(y) \cap E$ is denoted by $N(y,f,E)$.
Clearly (cf.~\cite[Proposition I.4.10]{Rickman}) we have
\begin{equation}\label{2d}
N(y,f,E)\leq n(E,y).
\end{equation}

The average value of $n(E,y)$ over the sphere $\partial B(0,t)$
is denoted by $\nu(E,t)$.
With the normalized $d$-dimensional Hausdorff measure ${\mathcal H}^d$
and $\omega_{d}={\mathcal H}^{d}(S^{d}(1))$ we thus have
\begin{equation}\label{2e}
\nu(E,t)=\frac{1}{\omega_{d-1}t^{d-1}}\int_{\partial B(0,t)} n(E,y)\;d{\mathcal H}^{d-1}(y).
\end{equation}
We will consider the case where $E$ is a ball
and put $\nu(r,t)=\nu\!\left(\overline{B}(0,r),t\right)$.
The following result is from~\cite[Lemma IV.1.1]{Rickman}.
\begin{lemma}\label{lemma3a}
Let $r,s,t>0$, $\theta>1$ and let $f\colon B(0,\theta r)\to\R^d$ be quasiregular.
Then
\[
\nu(r,t)\leq \nu(\theta r,s)+K_I(f)\frac{|\log(t/s)|^{d-1}}{(\log\theta)^{d-1}}.
\]
\end{lemma}
The next result is from~\cite[Theorem~3.1]{Bergw2013}.
\begin{lemma}\label{lemma3}
Let $E\subset \R^d$ be a set of positive capacity and $\theta>1$.
Then there exists a constant $c_1$ such that if $r>0$ and
$f\colon B(0,\theta r)\to\R^d\backslash E$ is quasiregular, then
$\nu(r,1)\leq c_1 K_I(f)$.
\end{lemma}
In~\cite[Theorem~3.1]{Bergw2013} this result is stated
with $\nu(r,1)$ replaced by $\nu(r,t)$ where $E\subset B(0,t/2)$, but
the above version follows easily from this using Lemma~\ref{lemma3a}.

\section{Proof of the theorems}
\begin{proof}[Proof of Theorem~\ref{thm1}]
Let $f\colon\R^d\to\R^d$ be a quasiregular map of transcendental type.
As already mentioned, the argument is essentially the same as
in~\cite[Theorem 1.3]{BN} where it was proved that $J(f)\subset \partial I(f)$.

Let
\[BO(f)=\left\{x\in \R^d\colon(f^n(x)) \text{ is bounded}\right\}\]
be the set of points with bounded orbits.
It was shown in~\cite[Theorem~1.4]{BN} that $BO(f)$ has
positive capacity.

Also, it was shown in~\cite[Theorem 1.2]{BDF} that $A(f)$ is non-empty
and that every component of $A(f)$ is unbounded. This
implies
that $A(f)$ contains continua and thus $A(f)$ has positive
capacity; see~\cite[Corollary III.2.5]{Rickman}.

Let now $x_0\in J(f)$ and let $U$ be a neighborhood of $x_0$.
By the
definition of $J(f)$, the complement of $\bigcup_{n=1}^\infty f^n(U)$ has
capacity zero. As both $A(f)$ and $BO(f)$ have positive capacity,
each of these sets intersects $\bigcup_{n=1}^\infty f^n(U)$.
Since  both $A(f)$ and $BO(f)$ are completely invariant,
we deduce that each of the two sets in fact intersects~$U$.
As $BO(f)\subset \R^d\backslash A(f)$ it follows that
$\partial A(f)$ intersects~$U$.
This holds for every neighborhood $U$ of~$x_0$.
Hence $x_0\in\partial A(f)$.
\end{proof}
\begin{proof}[Proof of Theorem~\ref{thm2}]
In view of Theorem~\ref{thm1} it suffices to prove that $\partial A(f)\subset J(f)$.
Suppose that  $x_0\in \partial A(f)$ but $x_0\notin J(f)$.
Since $J(f)$ is closed there exists $r>0$ such that
$B(x_0,4r)\cap J(f)=\emptyset$. Thus
the iterates of $f$ omit a set $C$ of positive capacity in $B(x_0,4r)$.
Since $x_0\in \partial A(f)$ there exist
$x_1\in B(x_0,r)\backslash A(f)$ and $x_2\in B(x_0,r)\cap A(f)$.
Let $R>2$ be such that $B(0,R)\cap J(f)\neq\emptyset$.
For $k\in\N$ we put
$F=f^k$,  $R_1=\max\{|F(x_1)|,R\}$ and $R_2=|F(x_2)|$.
For large $k$ we then have $R_2>R_1$.

For such $k$ we consider the sets
$X_1=\{x\in B(x_0,2r)\colon  |F(x)|\leq R_1\}$ and
$X_2=\{x\in B(x_0,2r)\colon  |F(x)|\geq R_2\}$.
Denote by $Y_j$ the component of $X_j$ that contains $x_j$, for $j\in\{1,2\}$.
By the maximum principle, $Y_2$ connects $x_2$ to $\partial B(x_0,2r)$.
We claim that, analogously, $Y_1$ connects $x_1$ to $\partial B(x_0,2r)$.
Suppose that this is not the case.
Then $Y_1$ is a compact subset of $B(x_0,2r)$.
This implies that $F(Y_1)=\overline{B}(0,R_1)$.
Hence $F(B(x_0,2r))\supset\overline{B}(0,R_1)\supset B(0,R)$.
Since $J(f)$ is completely invariant and  $B(0,R)\cap J(f)\neq\emptyset$,
we deduce that
$B(x_0,2r)\cap J(f)\neq\emptyset$, contradicting the choice of~$r$.

Let $\Gamma=\Delta(X_1,X_2;B(x_0,2r)$.
Using Lemma~\ref{lemma4} we find that
\begin{equation}
\begin{aligned}
\label{3a}
M(\Gamma)
&
\geq
M(\Delta(Y_1,Y_2;B(x_0,2r))
\\ &
\geq
c_0 \frac{\min\{\diam Y_1,\diam Y_2\}}{2r}
\geq \frac{c_0}{2}.
\end{aligned}
\end{equation}
On the other hand,
it follows from the proof of the $K_O$-inequality~\cite[Remark II.2.5]{Rickman}
that
\[
M(\Gamma)\leq K_O(F) \int_{\R^d}\rho(y)^d N\!\left(y,F,\overline{B}(x_0,2r)\right) dy
\]
for every $\rho\in{\mathcal F}(\Gamma)$.
We consider functions $\rho$ of the form $\rho(y)=\sigma(|y|)$ for
some continuous function $\sigma\colon [R_1,R_2]\to [0,\infty)$.
Then $\rho\in{\mathcal F}(\Gamma)$ if
\[
\int_{R_1}^{R_2}\sigma(t) dt\geq 1.
\]
Using~\eqref{2d} and \eqref{2e}
we obtain
\[
M(\Gamma)\leq
K_O(F)
\omega_{d-1}\int_{R_1}^{R_2} \sigma(t)^d \nu\!\left(\overline{B}(x_0,2r),t\right) t^{d-1} dt.
\]
For simplicity and without loss of generality we assume that
$x_0=0$ so that $\nu\!\left(\overline{B}(x_0,2r),t\right)=\nu(2r,t)$. Then the last inequality
takes the form
\begin{equation}\label{3b}
M(\Gamma)\leq
K_O(F)
\omega_{d-1}\int_{R_1}^{R_2} \sigma(t)^d \nu(2r,t) t^{d-1} dt.
\end{equation}
Lemmas~\ref{lemma3a} and \ref{lemma3} yield
\[
\nu(2r,t)
\leq  \nu(3r,1)+K_I(F) \left(\frac{\log t}{\log\frac32} \right)^{d-1}
\leq  K_I(F) \left(c_1+\left(\frac{\log t}{\log\frac32} \right)^{d-1}\right)
\]
and thus
\[
\nu(2r,t)
\leq c_2 K_I(F) \left(\log t\right)^{d-1}.
\]
for some constant $c_2$ and $t\geq R_1\geq 2$. Choosing
\[
\sigma(t)=\frac{1}{t\log t \log\displaystyle\frac{\log R_2}{\log R_1}}.
\]
we deduce from~\eqref{3b} that
\[
M(\Gamma)\leq
c_2\omega_{d-1} K_O(F) K_I(F)
\left(\log\frac{\log R_2}{\log R_1}\right)^{1-d}.
\]
Combining this with~\eqref{3a} yields
\[
\left(\log\frac{\log R_2}{\log R_1}\right)^{d-1}
\leq \frac{2c_2\omega_{d-1}}{c_0}
 K_O(F) K_I(F).
\]
Recall that $F=f^k$
and hence $K_O(F)\leq K_O(f)^k$
and $K_I(F)\leq K_I(f)^k$ by~\eqref{Kfn}.
With $K=(K_O(f)K_I(f))^{1/(d-1)}$ we deduce that
\[
\log\frac{\log R_2}{\log R_1} \leq c_3 K^k,
\]
with $c_3=(2c_2\omega_{d-1}/c_0)^{1/(d-1)}$.
Recall also that $R_1=\max\{|f^k(x_1)|,2\}$ and $R_2=|f^k(x_2)|$.
Since $x_2\in A(f)$ there exists $L\in \N$ such that $R_2\geq M^{k-L}(R,f)$
for $k\geq L$, with $R$ as in the definition of $A(f)$.
On the other hand,
$R_1\leq  M^{k-L-1}(R,f)$ for large~$k$.
Putting $S_k= M^{k-L-1}(R,f)$ we deduce that
\[
\log\frac{\log M(S_k,f)}{\log S_k} \leq c_3 K^k.
\]
By hypothesis, we have $\log M(S_k,f)\geq (\log S_k)^2$ for large~$k$.
This yields
\begin{equation}\label{3c}
\log \log S_k \leq c_3 K^k.
\end{equation}
On the other hand, using again the hypothesis \eqref{1a},
\[
\log \log S_k
=\log \log M(S_{k-1},f)
\geq 2 K \log\log S_{k-1}
\]
for large $k$ and hence
\[
\log \log S_k
\geq c_4 (2 K)^k
\]
for some positive constant $c_4$.
This contradicts~\eqref{3c}.
\end{proof}

\begin{remark}
The proof shows that \eqref{1a} can be replaced by
\[
\liminf_{r\to\infty}\frac{\log\log M(r,f)}{\log \log r}>
(K_I(f)K_O(f))^{1/(d-1)}.
\]
\end{remark}
\begin{remark}
It does not seem possible to improve the argument
by choosing a different function $\sigma$.
In fact, choosing $\rho(x)=c g(|x|)$ with $c=1/\int_{R_1}^{R_2} g(t)dt$
yields the best estimate when
\[
\displaystyle
\frac{ \int_{R_1}^{R_2} g(t)^d \left(t \log t\right)^{d-1} dt }
{\left(\int_{R_1}^{R_2} g(t)dt\right)^d}
\]
is minimal.
But by H\"older's inequality
we have
\begin{align*}
\int_{R_1}^{R_2} g(t)dt
&=
\int_{R_1}^{R_2} g(t) (t\log t)^{(d-1)/d}\frac{1}{(t\log t)^{(d-1)/d}}dt
\\ &
\leq
\left(\int_{R_1}^{R_2} g(t)^d (t\log t)^{d-1} dt\right)^{1/d}
\left(\int_{R_1}^{R_2} \frac{1}{t\log t}dt\right)^{(d-1)/d}
\\ &
=
\left(\int_{R_1}^{R_2} g(t)^d (t\log t)^{d-1} dt\right)^{1/d}
\left(\log\frac{\log R_2}{\log R_1}\right)^{(d-1)/d}.
\end{align*}
This shows that
$g(t)=1/(t\log t)$,
which corresponds to
our choice of~$\sigma$,
minimizes the expression in question.
\end{remark}

\end{document}